\documentclass[12pt]{article}
\usepackage{amsmath,amsfonts,amsthm,amssymb,mathtools,enumerate}
\usepackage{mathrsfs,graphicx,color,latexsym, tikz, calc}
\usepackage{hyperref} 
\usetikzlibrary{shadows}
\usetikzlibrary{patterns,arrows,decorations.pathreplacing}
\newtheorem{theorem}{Theorem}
\newtheorem{lemma}{Lemma}

\newtheorem{conjecture}{Conjecture}
\newtheorem{corollary}{Corollary}

\title{\bf \Large Spectral radius of graphs forbidden $C_7$ or $C_6^{\triangle}$}

\author{
{\small Junying Lu$^{a}$,\ \ Lu Lu$^{b,}$\footnote{Corresponding author.
\newline{ \hspace*{5mm}Email addresses:} lujunying\_math@163.com (J. Lu), lulugdmath@163.com (L. Lu), ytli0921@hnu.edu.cn (Y. Li).},\ \ Yongtao Li$^{c}$}\\[2mm]
\footnotesize $^a$ Department of Mathematics, Nanjing University, Nanjing 210093, China\\
\footnotesize $^b$ School of Mathematics and Statistics, Central South University, Changsha, Hunan, 410083, China\\
\footnotesize $^c$ School of Mathematics, Hunan University, Changsha, Hunan, 410082, China \\}

\date{ }

\begin{document}
\maketitle
\begin{abstract}
Let $C_k^{\triangle}$ be the graph obtained from a cycle $C_{k}$ by adding a new vertex connecting two adjacent vertices in $C_{k}$. In this note, we obtain the graph maximizing the spectral radius among all graphs with size $m$ and containing no subgraph isomorphic to $C_6^{\triangle}$. As a byproduct, we will show that if the spectral radius $\lambda(G)\ge1+\sqrt{m-2}$, then $G$ must contain all the cycles $C_i$ for $3\le i\le 7$ unless $G\cong K_3\nabla \left(\frac{m-3}{3}K_1\right)$.\\[1mm]

\end{abstract}

\baselineskip=0.202in

\section{Introduction}

For a simple graph $G = (V, E)$, we use $n= |V|$ and $m = |E|$ to denote the order and the size of $G$, respectively. Let $A(G)$ be the adjacency matrix of a graph $G$. The largest modulus of all eigenvalues of $A(G)$ is the spectral radius of $G$ and denoted by $\lambda(G)$. Denote the join of simple graphs $G$ and $H$ by $G\nabla H$.

Let $\mathcal{F}$ be a set of graphs, we say that $G$ is $\mathcal{F}$-free if it does not contain any element in $\mathcal{F}$ as a subgraph. Let $\mathcal{G}(m, \mathcal{F})$ denote the set of $\mathcal{F}$-free graphs with $m$ edges having no isolated vertices. When the forbidden set $\mathcal{F}$ is a singleton, say $\{F \}$, then we write $\mathcal{G}(m, F )$ for $\mathcal{G}(m, \mathcal{F})$. Brualdi and Hoffman \cite{Brualdi} posed a spectral Tur{\'a}n type problem: What is the maximal spectral radius of an $\mathcal{F}$-free graph of  given size $m$? Each of the corresponding extremal graphs is called the maximal graph. 
These extremal spectral graph problems attracted wide attention recently; 
see, for example, \cite{Niki2002cpc,Niki2009jctb} for 
$K_{r+1}$-free graphs, \cite{Niki2009laa} for $C_4$-free graphs, 
\cite{LP2022,LLH2022,Wang2022DM,ZS2021} for triangle-free non-bipartite graphs. 
Moreover, Zhai, Lin and Shu \cite{ZLS2021} studied this problem on $C_5$-free or $C_6$-free graphs with given size $m$, and left the following conjecture.

\begin{conjecture}[\cite{ZLS2021}]\label{con-1}
Let $k$ be a fixed positive integer and $G$ be a graph of sufficiently large size $m$ without
isolated vertices. If $\lambda(G)\ge \frac{k-1+\sqrt{4m-k^2+1}}{2 }$, then $G$ contains a cycle of length $t$ for every $t \le 2k + 2$, unless $G\cong K_k\nabla (\frac{m}{k}-\frac{k-1}{2})K_1$.
\end{conjecture}

An equivalent version of Zhai--Lin--Shu's conjecture can be stated as 

\begin{conjecture}[Zhai--Lin--Shu] \label{Conj-ZLS}
Let $k\ge 2$ and $G$ be a graph of sufficiently large size $m$ 
without isolated vertices. 
If $G$ is  $C_{2k+1}$-free or $C_{2k+2}$-free, then 
\[  \lambda (G)\le \frac{k-1 +\sqrt{4m -k^2+1}}{2}, \]
equality holds if and only if  
$G\cong K_k \nabla (\frac{m}{k} - \frac{k-1}{2})K_1$. 
\end{conjecture}

In 2021, Zhai, Lin and Shu \cite{ZLS2021} proved the conjecture in the case 
$k=2$. 
Note in this case that the extremal graph $K_2\nabla \frac{m-1}{2}K_1$ is 
 well-defined only in the case of odd $m$. 
Later, Min, Lou and Huang \cite{MLH2021} 
proved the case of $k=2$ and even $m$. 
Namely, by adding a pendant edge to a maximum-degree vertex of 
$K_2\nabla \frac{m-2}{2}K_1$.  
Furthermore, stability-type results involving the case $k=2$ were 
proved by Li, Sun and Wei \cite{LSW2022} recently.  
It is worth noting that Conjecture \ref{Conj-ZLS} remains open for the case $k\ge 3$. 
In this paper, we will show the case $k=3$ for $C_7$-free graphs.

In fact, we will prove a more slightly general result.  
Let $C_t^{\triangle}$ denote the graph on $t+1$ vertices obtained from 
$C_t$ and $C_3$ by identifying an edge. In other words, 
$C_t^{\triangle}$ can be obtained from $C_t$ by adding a new vertex and joining 
this vertex to two adjacent vertices of $C_t$. 
Clearly, we can see that both $C_t$ and $C_{t+1}$ are subgraphs 
of $C_t^{\triangle}$. 
It was proved by Zhai, Lin and Shu in \cite{ZLS2021} that 
the complete bipartite graphs attain the maximum spectral radius 
when both the substructures  $C_3^{\triangle}$ and $C_4^{\triangle}$ are forbidden. 
 In \cite{Nik2021}, Nikiforov showed that if $G$ is a graph with $m$ edges and $\lambda^2(G)\ge m$, then the maximum number of triangles with a common edge in $G$ is greater than $\frac{1}{12}\sqrt[4]{m}$, unless $G$ is a complete bipartite graph with possibly some isolated vertices. As a conclusion, the complete bipartite graphs attain the maximum spectral radius when we only forbid $C_3^{\triangle}$. 
 Very recently, Li, Sun and Wei \cite{LSW2022} 
 determined that $K_2\nabla \frac{m-1}{2}K_1$ is the unique extremal graph for $C_4^{\triangle}$-free or $C_5^{\triangle}$-free graphs when the size $m$ is odd. 
 Soon after, Fang, You and Huang \cite{FYH2022} further determined 
 the extremal graph for even $m$. 
The following conjecture was recently proposed in \cite{LLH2022}. 

\begin{conjecture}[\cite{LLH2022}] \label{Conj-L}
Let $k\ge 2$ and $G$ be a graph of sufficiently large size $m$ 
without isolated vertices. 
If $G$ is  $C_{2k}^{\triangle}$-free or $C_{2k+1}^{\triangle}$-free, then 
\[  \lambda (G)\le \frac{k-1 +\sqrt{4m -k^2+1}}{2}, \]
equality holds if and only if  
$G\cong K_k \nabla (\frac{m}{k} - \frac{k-1}{2})K_1$. 
\end{conjecture}

Motivated by the previous works \cite{FYH2022,LSW2022,MLH2021,ZLS2021}, 
 we will verify in this paper that Conjecture \ref{Conj-L} holds 
 for the case $k=3$ and 
 we characterize the unique graph with the maximum spectral radius among $\mathcal{G}(m,C_6^{\triangle})$.

\begin{theorem}\label{thm-1}
Let $G$ be a graph in $\mathcal{G}(m,C_6^{\triangle})$. If $m\ge 38$, then 
\[ \lambda(G)\le 1+\sqrt{m-2} . \]
Moreover, the equality holds if and only if $G\cong K_3\nabla \frac{m-3}{3}K_1$.
\end{theorem}

Since $C_7$ is a subgraph of $C_6^{\triangle}$, we have $\mathcal{G}(m,C_7)\subseteq \mathcal{G}(m,C_6^{\triangle})$. Combining Theorem 1.5 of \cite{ZLS2021}, we obtain the following result.

\begin{corollary}\label{cor-1}
Let $G$ be a graph with size $m\ge 38$. If the spectral radius 
\[ \lambda(G)\ge1+\sqrt{m-2}, \] 
then $G$ must contain all the cycles $C_i$ for $3\le i\le 7$ unless $G\cong K_3\nabla \left(\frac{m-3}{3}K_1\right)$
\end{corollary}

Corollary \ref{cor-1} proves a very special case of Conjecture \ref{con-1}.
\section{Preliminaries}\label{sec-2}
For a graph $G$ and a subset $S\subseteq V(G)$, let $G[S]$ denote the subgraph of $G$ induced by $S$. Let $e(G)$ denote the size of $G$. For two vertex subsets $S$ and $T$ of $G$ (where $S\cap T$ may not be empty), let $e(S,T)$ denote the number of edges with one endpoint in $S$ and the other in $T$. The notation $e(S,S)$ is simplified by $e(S)$. For a vertex $v\in V(G)$, let $N(v)$ be the neighborhood of $v$, $N[v]=N(v)\cup \{v\}$ and $N^2(v)$ be the set of vertices of distance two to $v$. In particular, let $N_S(v)=N(v)\cap S$ and $d_S(v)=|N_S(v)|$.

It is known that $A(G)$ is irreducible and nonnegative for a connected graph $G$. From the Perron-Frobenius Theorem, there is a unique positive unit eigenvector corresponding to $\lambda(G)$, which is called the \emph{Perron vector} of $G$. Let $\mathbf{x}$ be the Perron vector of $G$ with coordinate $x_v$ corresponding to the vertex $v\in V(G)$. A vertex $u$ in $G$ is said to be an \emph{extremal vertex} if $x_u=\max\{x_v\colon v\in V(G)\}$.

A \emph{cut vertex} of a graph is a vertex whose deletion increases the number of components. A graph is called
\emph{$2$-connected}, if it is a connected graph without cut vertices.

\begin{lemma}[\cite{ZLS2021}]\label{lem-x-3}
Let $G$ be a graph in $\mathcal{G}(m, F)$ such that $\lambda(G)$ is as large as possible, where $F$ is $2$-connected. Then $G$ is connected, and there exists no cut vertex in $V(G)\setminus \{u^*\}$ where $u^*$ is an extremal vertex of $G$.
\end{lemma}

For $n\ge k>2s >0$, let $G_{n,k,s}=(K_{k-2s}\cup(n-k+s)K_1 )\nabla K_s$.
\begin{lemma}[\cite{Balister}]\label{lem-x-5}
Let $G$ be a connected graph on n vertices containing no path on $k+1$ vertices, $n > k\ge 3$. Then 
\[ e(G) \le \max \left\{\binom{k-1}{2} + (n - k + 1), 
\binom{\left \lceil \frac{k+1}{2} \right \rceil }{2} + 
\left \lfloor \frac{k-1}{2} \right \rfloor \left(n-\left \lceil \frac{k+1}{2} \right \rceil \right) \right\}. \] 
 If equality
occurs then $G$ is either $G_{n,k,1}$ or $G_{n,k,\left \lfloor (k-1)/2 \right \rfloor}$.
\end{lemma}

Note that if $G$ is $C_6^{\triangle}$-free, then 
for every $v\in V(G)$, we know that 
the induced subgraph $G[N(v)]$ is $P_6$-free. 
The following lemma gives  clearly the structure of 
induced $P_6$-free graphs. 
A set $U\subseteq V$ \emph{dominates} a set $U'\subseteq V$ if any vertex $v \in U'$ either lies in $U$ or has a neighbor in $U$. We also say that $U$ dominates $G[U']$. A subgraph $H$ of $G$ is a \emph{dominating subgraph} of $G$ if $V(H)$ dominates $G$.
\begin{lemma}[\cite{Liu}]\label{lem-x-6}
A graph $G=(V,E)$ contains no induced subgraph isomorphic to $P_6$ if and only if each connected induced subgraph of G contains a dominating induced $C_6$ or a dominating (not necessarily induced) complete bipartite graph.
\end{lemma}

\section{Proof}\label{sec-3}
Let $G^*$ be a graph in $\mathcal{G}(m,C_6^{\triangle})$ with the maximum spectral radius. In the view of Lemma \ref{lem-x-3}, we know that $G^*$ is connected. Assume that $\mathbf{x}$ is the Perron vector of $G^*$ and let $u^*$ be the extremal vertex of $G^*$. Since $K_3\nabla \frac{m-3}{3}K_1$ is $C_6^{\triangle}$-free, one has
\[\lambda :=\lambda(G^*)\ge \lambda(K_3\nabla \tfrac{m-3}{3}K_1)=1+\sqrt{m-2}.\]

Denote $W=V(G^*)\setminus (N(u^*)\cup \{u^*\})$. Let $N_0(u^*)$ 
be the isolated vertices of the induced subgraph $G^*[N(u^*)]$, 
and $N_+(u^*)=N(u^*)\setminus N_0(u^*)$ be the vertices 
of $N(u^*)$ with degree at least one in $G^*[N(U^*)]$. 
Since $A(G^*)\mathbf{x}=\lambda\mathbf{x}$, we have
\begin{equation}\label{eq1}
\lambda x_{u^*}=\sum_{u\in N_0(u^*)}x_u+\sum_{u\in N_+(u^*)}x_u.
\end{equation}
In addition, we also have $A^2(G)\mathbf{x}={\lambda}^2\mathbf{x}$. It follows that
\begin{align}\label{eq2}
{\lambda}^2x_{u^*}=\sum_{u\in V(G^*)}a_{u^*u}^{(2)}x_u
=d(u^*)x_{u^*}+\sum_{u\in N_+(u^*)}d_{N(u^*)}(u)x_u+\sum_{w\in N^2(u^*)}d_{N(u^*)}(w)x_w,
\end{align}
where $a_{vu}^{(2)}$ denotes the number of walks of length 2 starting 
from $v$ to $u$.

Note that $\lambda\ge 1+\sqrt{m-2}\ge 7$ whenever $m\ge 38$. Then 
\[ m-3 \le \lambda^2-2\lambda.\]
 In view of (\ref{eq1}) and (\ref{eq2}), one has
{\small\begin{align}
&(m-3)x_{u^*}\le (\lambda^2-2\lambda)x_{u^*} \notag\\[2mm]
&=|N(u^*)|x_{u^*}+\sum_{u\in N_+(u^*)}(d_{N(u^*)}(u)-2)x_u+\sum_{w\in N^2(u^*)}d_{N(u^*)}(w)x_w-2\sum_{u\in N_0(u^*)}x_u \label{eq3}\\
&\le |N(u^*)|x_{u^*}+\sum_{u\in N_+(u^*)}(d_{N(u^*)}(u)-2)x_u 
+e(W,N(u^*)) x_{u^*} -2\sum_{u\in N_0(u^*)}x_u. \notag
\end{align}}
Since $m=|N(u^*)| + e(N_+(u^*)) + e(W) + e(W,N(u^*))$, it follows that
\begin{align}\label{eq}
 e(W)\le \sum_{u\in N_+(u^*)}(d_{N(u^*)}(u)-2)x_u/x_{u^*}-e(N_+(u^*))-2\sum_{u\in N_0(u^*)}x_u/x_{u^*}+3  .
\end{align}
Moreover,  the  above equality holds if and only if $\lambda^2-2\lambda=m-3$, and $x_w=x_{u^*}$ 
for every $w\in W$ satisfying $ d_{N(u^*)}(w)\ge 1$.

For each non-trivial connected component $H$ of $G^*[N(u^*)]$, we denote 
\[  \eta(H) :=\sum_{u\in V(H)}(d_H(u)-2)x_u/x_{u^*}-e(H). \] 
 Then (\ref{eq}) gives 
\begin{align}\label{eq4}
e(W)\le \sum_H\eta(H)-2\sum_{u\in N_0(u^*)}x_u/x_{u^*}+3,
\end{align}
where $H$ varies over all non-trivial components of $G^*[N(u^*)]$.

We would like to emphasize
that $\eta(H)$ is an important parameter in this paper, 
and it plays a significant role in the proof of Theorem \ref{thm-1}.
First of all, in the forthcoming Lemma \ref{lem-1} and Lemma \ref{lem-2}, 
we will show that $\eta(H)$ is a nonpositive value, and we also provide some upper bounds on
$\eta (H)$, where $H$ is a non-trivial component of $G^*$. 
Moreover, we will use these upper bounds to show that $G^*[N(u^*)]$ contains exactly one component $H$ and then 
we characterize the structure of this component; 
see Lemmas \ref{lem-3}, \ref{lem-4} and \ref{lem-8} for more details.  
 
\begin{lemma}\label{lem-1}
Let $H$ be a non-trivial component of $G^*[N(u^*)]$. If $\delta(H)\ge 2$, then 
\[  \eta (H) \le \begin{cases}
 0, & \text{if $H=K_5$,}\\
 -1,& \text{if $H=K_5-e$,} \\
 -2, & \text{if $H=K_4$ or $K_5-2e$,} \\
 -3, & \text{otherwise,}
\end{cases}  \]
where $K_5-e$ denotes the graph obtained from $K_5$ by deleting one edge and $K_5-2e$ denotes the graph obtained from $K_5$ by deleting any two edges.
\end{lemma}

\begin{proof}
Since $\delta(H)\ge 2$ and $x_u/x_{u^*} \le 1$, we have
\begin{align}\label{eq'}
    \eta(H)=\sum_{u\in V(H)}(d_H(u)-2)x_u/x_{u^*}-e(H)\le e(H)-2|V(H)|
\end{align}
and equality holds only if $x_u=x_{u^*}$ for all $u\in V(H)\setminus\{x\in V(H)\colon d_H(x)=2\}$.

Note that $H$ is $P_6$-free since $G^*$ is $C_6^{\triangle}$-free. Let $h=|V(H)|$. If $h\ge 6$ then we have
\[e(H)\le \max\{h+2,2h-3\}=2h-3\]
from Lemma \ref{lem-x-5}. It follows that
\[\eta(H)\le e(H)-2h\le 2h-3-2h=-3,\]
and the equality occurs only if $H=G_{h,5,2}=K_2\nabla (h-2)K_1$. 

In addition, for $h=3$ or $4$, we have 
\[\eta(H)\le \binom{h}{2}-2h=\left\{\begin{array}{cc}
    -3 & \text{if } h=3,  \\
    -2 & \text{if } h=4,
\end{array}\right.\]
where each equality holds if and only if $H$ is complete. For $h=5$,
\[\eta(H)\le \left\{\begin{array}{cl}
    0 & \text{if } H=K_5,\\
    -1 & \text{if } H=K_5-e,  \\
    -2 & \text{if } H=K_5-2e,\\
    -3 &\text{otherwise}.
\end{array}\right.\]

The proof is completed.
\end{proof}

\begin{lemma}\label{lem-2}
Let $H$ be a non-trivial component of $G^*[N(u^*)]$. If $\delta(H)=1$ then
\[\eta(H)<\left\{\begin{array}{cc}
    -1, & \text{if } H=K_2, \\
    -2, & \text{otherwise}.
\end{array}\right.\]
\end{lemma}
\begin{proof}
Firstly, suppose that $H$ is a star, i.e., $H=K_{1,t}$ with $v$ being the non-pendant vertex and $v_1,\ldots,v_t$ being pendant vertices. If $t=1$, then $\eta(H)=-\sum_{u\in V(H)}x_u/x_{u^*}-1<-1$. If $t\ge 2$, then $\eta(H)=(t-2)x_v/x_{u^*}-\sum_{i=1}^tx_{v_i}/x_{u^*}-t\le(t-2)-\sum_{i=1}^tx_{v_i}/x_{u^*}-t<-2$.

Suppose $H$ is non-star. Since $H$ is $P_6$-free, $H$ contains a dominating (not necessarily induced) complete bipartite graph due to Lemma \ref{lem-x-6}. Assume that $H[S,T]=K_{s,t}$ is such a maximal complete bipartite graph. Denote by $X=V(K_{s,t})$ and $Y=V(H)\setminus X$. Without loss of generality, suppose $s\le t$. Note that $K_{3,3}$ contains $P_6$, we have $s\le 2$.

{\flushleft\bf Case 1.} $s=2$.

In this case, $X$ contains no pendant vertex. Thus, $Y\ne \emptyset$ and there is a pendant vertex in $Y$. Furthermore, $H[Y]$ contains no $K_2$ since $H$ is $P_6$-free.
\begin{figure}[htpb]
    \centering
    \includegraphics[width=13cm]{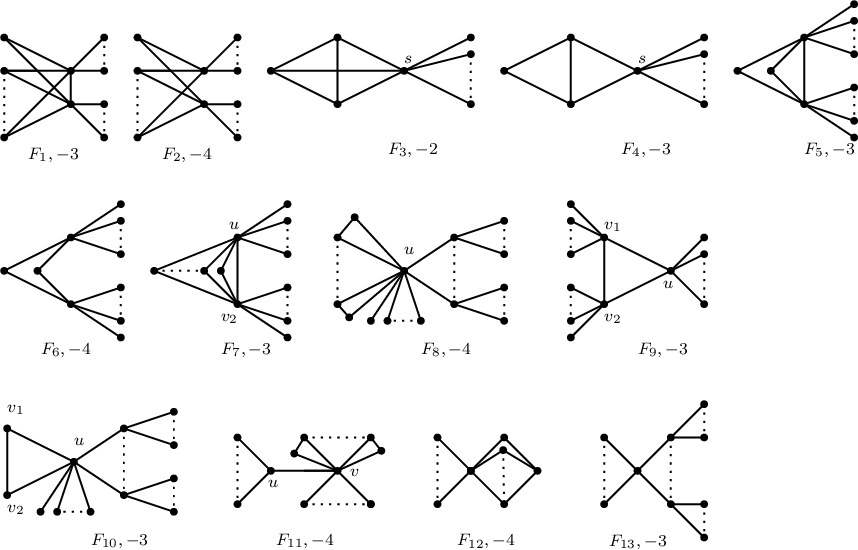}
    \caption{The graphs $F_1, F_2, \ldots, F_{13}$ used in the proofs, with graphs labelled as they appear alongside strict upper bounds for $\eta(F_i)$.}
    \label{fig-1}
\end{figure}
If $t\ge 3$ then $H[T]$ contains no $K_2$ and $N(T)\cap Y=\emptyset$ because $H$ is $P_6$-free. Hence, for any $y\in Y$, $N_H(y)\subseteq S$. If there is a vertex $y\in Y$ such that $N_H(y)=S$, then $H[S,T\cup\{y\}]$ is a complete bipartite graph dominating $H$, which contradicts the maximality of $H[S,T]$. Therefore, for any $y\in Y$, $d_H(y)=1$. It follows that $H=F_1$ if $H[S]$ contains $K_2$, and $H=F_2$ if $H[S]$ contains no $K_2$ (see Figure \ref{fig-1}). By simple calculations, we have
\[\eta(F_1)=\sum_{u\in X}(d_{H}(u)-2)x_u/x_{u^*}+\sum_{v\in Y}(d_H(v)-2)x_v/x_{u^*}-e(H)\le -3-\sum_{v\in Y}x_v/x_{u^*}<-3,\]
and
\[\eta(F_2)=\sum_{u\in X}(d_{H}(u)-2)x_u/x_{u^*}+\sum_{v\in Y}(d_H(v)-2)x_v/x_{u^*}-e(H)\le -4-\sum_{v\in Y}x_v/x_{u^*}<-4.\]

If $t=2$ then $|N(y)\cap S|,|N(y)\cap T|\le1$ for any $y\in Y$ due to the maximality of $H[S,T]$. Since $H$ is $P_6$-free, one can easily verify that $N(Y)\cap T=\emptyset$ if $N(Y)\cap S\ne \emptyset$. Without loss of generality, assume that $N(Y)\cap S\ne\emptyset$, that is, $N_H(y)\in S$ for any $y\in Y$. Since $H$ is $P_6$-free, we have $N(Y)=\{s\}$ for a fixed vertex $s\in S$ when $T$ contains an edge. Therefore, if $T$ contains an edge then $H=F_3$ or $F_4$, and $H=F_5$ or $F_6$ otherwise (see Figure \ref{fig-1}). By similar calculations, we have
\[\eta(F_3)<-2,\eta(F_4)<-3,\eta(F_5)<-3 \mbox{ and }\eta(F_6)<-4.\]
Thus, $\eta(H)<-2$.

{\flushleft\bf Case 2.} $s=1$.

Let $S=\{u\}$. By the maximality of $H[S,T]$, we may assume $t\ge 2$ and $N(Y)\cap S=\emptyset$. Since $H$ is $P_6$-free and $\delta(H)=1$, one can verify that $H[T]$ is $P_4$-free.

{\bf Subcase 2.1.} $H[T]$ contains $P_3$.

Assume $H[T]$ contains $P_3$, say $P_3=v_1v_2v_3$, then $H[Y]$ contains no $K_2$ because $H$ is $P_6$-free. If $v_1\sim v_3$ then $H[u,v_1,v_2,v_3]=K_4$. Therefore, $H$ has the form of $F_3$ shown in Figure \ref{fig-1}. Now, suppose $v_1\not\sim v_3$. If $t=3$ then $H[X]=K_{2,2}+e$ and thus $H$ has the form of $F_5$. If $t\ge 4$ then $N(v)\cap Y=\emptyset$ for all $v\in T\setminus\{v_2\}$ because $H$ is $P_6$-free. Note that $H[T\setminus\{v_1,v_2,v_3\}]$ contains no $K_2$, $H$ has the form of $F_7$. Similarly, $\eta(F_7)\le -3-\sum_{y\in H_0} x_y/x_{u^*}<-3$ where $H_0=\{v\in H\colon d_H(v)=1\}$, and thus $\eta(H)<-3$.

{\bf Subcase 2.2.} $H[T]$ contains no $P_3$ but $P_2$.

Suppose $H[T]$ contains $2K_2$. Since $H$ is $P_6$-free,  we have $N(Y)\subseteq T_0$ and $Y\subseteq H_0$, where $T_0=\{v\in T\colon  d_T(v)=0\}$. Thus, $H=F_8$. Let $k$ be the matching number of $H[T]$ and $l$ be the size of $\{v\in T_0\colon N_Y(v)\ne \emptyset\}$, we have $\eta(F_8)\le -k-l-2-\sum_{y\in H_0}x_y/x_{u^*}<-4$.

Suppose $H[T]$ contains exactly one $K_2$, say $v_1v_2$. Therefore, we claim $t\ge 3$. Otherwise, there is $v_i$ such that $d_H(v_i)\ge 3$. Take $v_i$ as the center of star, we have $K_{1,t'}(t'\ge 3>t)$ dominates $H$, which contradicts the maximality of $H[S,T]$. Note that $H[Y]$ contains no $K_2$. Since $H$ is $P_6$-free, we have
\begin{enumerate}[(i)]
    \item $|N_Y(v_1)\cap N_Y(v_2)|\le 1$;
    \item for any $y\in Y$, if $N(y)\cap (T\setminus\{v_1,v_2\})\ne \emptyset$ then $d_T(y)=1$;
    \item if $N(v_i)\cap Y\ne \emptyset(i=1,2)$ then $N(T\setminus\{v_1,v_2\})\cap Y=\emptyset$.
\end{enumerate}
Now, if there is $y\in Y$ such that $d_T(y)\ne 1$ then $H$ has the form of $F_4$. Otherwise $H$ has the form of $F_9$ or $F_{10}$ (see Figure \ref{fig-1}). One can easily check $\eta(H)<-3$ in all such cases.

{\bf Subcase 2.2.} $H[T]$ contains no $K_2$.

Assume $H[T]$ contains no $K_2$, then $H[Y]$ contains no $P_3$. Suppose that there is an edge $y_1y_2\in H[Y]$. We have $d_T(y_1)=d_T(y_2)=1$ and $N_T(y_1)=N_T(y_2)$. Assume that $y_i\sim v$ for some $v\in T$. We have $N(T\setminus\{v\})\cap Y=\emptyset$. Thus $H=F_{11}$ and $\eta(F_{11})<-4$ by similar calculations. Now consider that $H[Y]$ contains no edge. If there is $y\in Y$ such that $d_T(y)\ge 2$ then $N(T\setminus N(y))\cap Y=\emptyset$. Therefore, if $T=N(y)$ then the complete bipartite graph $H[S\cup\{y\},T]$ dominates $H$, which contradicts the maximality of $H[S,T]$; if $T\setminus N(y)\ne \emptyset$ then $N_Y(N(y))=\{y\}$ and thus $H=F_{12}$. One can verify $\eta(F_{12})<-4$. Now suppose $d_T(y)=1$ for all $y\in Y$. Then $H$ is a tree having the form of $F_{13}$ and thus $\eta(H)<-3$.

The proof is completed.
\end{proof}

We claim there exists a non-trivial component in $G^*[N(u^*)]$. Otherwise, 
if $N_+(u^*)=\emptyset$, then $G^*[N[u^*]]$ is a star. It yields that $0\le e(W) 
\le - 2\sum_{u\in N_0(u^*)}x_u/x_{u^*}+3$ from \eqref{eq4}. Therefore, $\lambda x_{u^*}=\sum_{u\in N_0(u^*)}x_u\le \frac{3}{2}x_{u^*}$, 
which implies $\lambda\le \frac{3}{2}<1+\sqrt{m-2}$ since $m\ge 3$, a contradiction. Hence, the claim holds. 

\begin{lemma}\label{lem-6}
For any non-trivial component $H$ in $G^*[N(u^*)]$, we have $H\ne K_5$ or $K_5-e$.
\end{lemma}
\begin{proof}
Suppose to the contrary that there is a component $H=K_5$ in $G^*[N(u^*)]$. For any component (not necessarily non-trivial) $H'$ in $G^*[N(u^*)]$, since $G^*$ is $C_6^\triangle$-free, we have $N_W(H)\cap N_W(H')=\emptyset$ and there is no path of length less than 3 between $N_W(H)$ and $N_W(H')$. Note that $e(W)\le \sum_H\eta(H)-2\sum_{u\in N_0(u^*)}x_u/x_{u^*}+3\le3$. For any $w_1,w_2 \in N_W(H)$ in a common path, we have $N_H(w_1)=N_H(w_2)$, since otherwise, there will exist a $C_6^\triangle$ in $G^*$. Let $V(H)=\{u_1,u_2,u_3,u_4,u_5\}$. If $N(W)\cap H=\emptyset$ then $x_{u_1}=x_{u_2}=x_{u_3}=x_{u_4}=x_{u_5}$. Now assume that $N(W)\cap H\ne \emptyset$, then $d_H(w)\le 1$ for any $w\in N_W(H)$ because $G^*$ is $C_6^\triangle$-free. From Lemma \ref{lem-x-3}, there exists no cut vertex in $V(G^*)\setminus \{u^*\}$ and thus $N_W(w)\ne \emptyset$ for $w\in N_W(H)$. It indicates that there exists a path $P$ between $N_W(H)$ and $N_W(H')$ for some component $H'$ in $G^*[N(u^*)]$. Otherwise, for any component $W'$ in $W$, if $N_H(W')\ne \emptyset$ then there exists $v\in H$ such that $N_H(W')=\{v\}$. Therefore, there will be a cut vertex in $H$ whenever $N_H(W)\ne \emptyset$, which contradicts that $u^*$ is the only possible cut vertex. By noticing that $e(W)\le3$, the length of $P$ is 3 and thus $d_W(w)=d_H(w)=0$ for any $w\in W\setminus V(P)$. Let $u_5$ be the vertex adjacent to one of the endpoints of $P$, then $N(u_1,u_2,u_3,u_4)\cap W=\emptyset$. Therefore, $x_{u_1}=x_{u_2}=x_{u_3}=x_{u_4}$ still holds. Hence, we always have $\lambda x_{u_1}=x_{u^*}+3x_{u_1}+x_{u_5}$. Thus $x_{u_1}\le \frac{2}{\lambda-3}x_{u^*}\le \frac{1}{2}x_{u^*}$ for $\lambda \ge 7$ and 
\begin{align*}
    e(W)&\le \sum_{H'}\eta(H')-2\sum_{u\in N_0(u^*)}x_u/x_{u^*}+3\\
    &=\eta(H)+ \sum_{H'\ne H}\eta(H')-2\sum_{u\in N_0(u^*)}x_u/x_{u^*}+3\\
    &\le 8 x_{u_1}/x_{u^*}+2x_{u_5}/x_{u^*}-10+\sum_{H'\ne H}\eta(H')+3\\
    &\le -1,
\end{align*}
a contradiction.
Similarly, we can prove that $H\ne K_5-e$ for any non-trivial component $H$ in $G^*[N(u^*)]$.
\end{proof}
By Lemmas \ref{lem-1}, \ref{lem-2} and \ref{lem-6}, we get 
\[e(W)\le \sum_H\eta(H)- 2\sum_{u\in N_0(u^*)}x_u/x_{u^*}+3<-1- 2\sum_{u\in N_0(u^*)}x_u/x_{u^*}+3\le 2.\]
Since $e(W)$ is an integer, we obtain $e(W)\le 1$.

\begin{lemma}\label{lem-3}
If $W\neq \emptyset$, then $e(W)=0$.
\end{lemma}

\begin{proof}
Suppose to the contrary that $e(W)=1$. Then
\[\sum_H\eta(H)\ge e(W)+ 2\sum_{u\in N_0(u^*)}x_u/x_{u^*}-3 
\ge -2+2\sum_{u\in N_0(u^*)}x_u/x_{u^*}\ge -2.\]
By Lemmas \ref{lem-1}, \ref{lem-2}and \ref{lem-6}, there is exactly one non-trivial component $H$ in $G^*[N(u^*)]$ and $H=K_2$, $K_4$ or $K_5-2e$.
If $H=K_4$ or $K_5-2e$, then $-2 \ge \eta(H)\ge -2+ 2\sum_{u\in N_0(u^*)}x_u/x_{u^*}$. It yields that $N_0(u^*)=\emptyset$ and $G^*[N(u^*)]=H=K_4$ or $K_5-2e$. So 
$x_u=x_{u^*}$ for all $u\in V(H)$. Thus, $\lambda x_{u^*}=\sum_{u\in V(H)}x_u\le 5x_{u^*}$, which implies $\lambda=5<1+\sqrt{m-2}$ for any $m\ge 19$, a contradiction. Therefore, $H=K_2$ and 
\begin{align*}
0&\le e(W)\le \eta(H)-2\sum_{u\in N_0(u^*)}x_u/x_{u^*}+3\\
&\le -1-\sum_{u\in V(H)}x_u/x_{u^*}-2\sum_{u\in N_0(u^*)}x_u/x_{u^*}+3\\
&=2-\sum_{u\in V(H)}x_u/x_{u^*}-2\sum_{u\in N_0(u^*)}x_u/x_{u^*}.
\end{align*}
Thus,
\[
\lambda x_{u^*}= \sum_{u\in V(H)} x_u+\sum_{u\in N_0(u^*)}x_u\le \sum_{u\in V(H)} x_u+2\sum_{u\in N_0(u^*)}x_u\le 2x_{u^*}.
\]
It follows that $\lambda\le 2<1+\sqrt{m-2}$, a contradiction.
\end{proof}

\begin{lemma}\label{lem-4}
$G^*[N(u^*)]$ contains exactly one non-trivial component $H$ and $\delta(H)\ge 2$.
\end{lemma}

\begin{proof}
Lemma \ref{lem-3} gives $e(W)=0$, we have $W=N^2(u^*)$. From \eqref{eq4}, we have \[\sum_H\eta(H)\ge -3+2\sum_{u\in N_0(u^*)}x_u/x_{u^*}\ge -3.\]
It yields that $G^*[N(u^*)]$ contains at most two non-trivial components.

Suppose there are two non-trivial components $H_1,H_2$ in $G^*[N(u^*)]$. 
Then $H_1=H_2=K_2$ due to Lemmas \ref{lem-1} and \ref{lem-2}. We have
\begin{align*}
0&= e(W)\le \eta(H_1)+ \eta(H_2)-2\sum_{u\in N_0(u^*)}x_u/x_{u^*}+3\\
&< -2-\sum_{u\in V(H_1\cup H_2)}x_u/x_{u^*}-2\sum_{u\in N_0(u^*)}x_u/x_{u^*}+3\\
&=1-\sum_{u\in V(H_1\cup H_2)}x_u/x_{u^*}-2\sum_{u\in N_0(u^*)}x_u/x_{u^*}.
\end{align*}
Thus,
\[
\lambda x_{u^*}= \sum_{u\in V(H_1\cup H_2)} x_u+\sum_{u\in N_0(u^*)}x_u\le \sum_{u\in V(H_1\cup H_2)} x_u+2\sum_{u\in N_0(u^*)}x_u< x_{u^*}.
\]
It follows that $\lambda< 1<1+\sqrt{m-2}$, a contradiction. Therefore, there is exactly one non-trivial component $H$ in $G^*[N(u^*)]$.

Next, we show $\delta(H)\ge 2$. Otherwise, $H=K_{1,s}$ with $s\ge 2$, or $H\cong F_3$ due to Lemma \ref{lem-2}. Firstly we consider $H=K_{1,s}$. Let $u$ be the center of $H$ and $\{v_1,v_2,\ldots,v_s\}=H\setminus \{u\}$.
 We have
\begin{align*}
0&= e(W)\le \eta(H)-2\sum_{u\in N_0(u^*)}x_u/x_{u^*}+3\\
&\le -2-\sum_{i=1}^s x_{v_i}/x_{u^*}-2\sum_{u\in N_0(u^*)}x_u/x_{u^*}+3\\
&=1-\sum_{i=1}^s x_{v_i}/x_{u^*}-2\sum_{u\in N_0(u^*)}x_u/x_{u^*}.
\end{align*}
Thus,
\[
\lambda x_{u^*}= x_u+\sum_{i=1}^s x_{v_i}+\sum_{u\in N_0(u^*)}x_u\le x_{u^*}+\sum_{i=1}^s x_{v_i}+2\sum_{u\in N_0(u^*)}x_u\le x_{u^*}+x_{u^*}=2x_{u^*}.
\]
It follows that $\lambda\le 2<1+\sqrt{m-2}$, a contradiction. Now consider the case of $H=F_3$. Let $V(H)=\{u_1,u_2,u_3,u_4\}\cup \{v_1,v_2,\ldots,v_k\}$, where $d_H(u_1)=k+3$ and $H[u_1,u_2,u_3,u_4]=K_4$. Then we have,
\begin{align*}
0&\le e(W)\le \eta(H)-2\sum_{u\in N_0(u^*)}x_u/x_{u^*}+3\\
&\le -2-\sum_{i=1}^kx_{v_i}/x_{u^*}-2\sum_{u\in N_0(u^*)}x_u/x_{u^*}+3\\
&=1-\sum_{i=1}^kx_{v_i}/x_{u^*}-2\sum_{u\in N_0(u^*)}x_u/x_{u^*}.
\end{align*}
Thus,
\begin{align*}
\lambda x_{u^*}&= x_{u_1}+x_{u_2}+x_{u_3}+x_{u_4}+\sum_{i=1}^k x_{v_i}+\sum_{u\in N_0(u^*)}x_u\\
&\le 4x_{u^*}+\sum_{i=1}^k x_{v_i}+2\sum_{u\in N_0(u^*)}x_u \le 4x_{u^*}+x_{u^*}=5x_{u^*}.
\end{align*}
It follows that $\lambda=5<1+\sqrt{m-2}$ since $m\ge 27$, a contradiction.

The proof is completed.
\end{proof}

By Lemma \ref{lem-4}, we assume that 
$H$ is the unique non-trivial component in $G^*[N(u^*)]$. 

\begin{lemma}\label{lem-8} 
$\eta (H)=-3$, $N_0(u^*)= \emptyset$ and 
$G^*[N(u^*)]=K_{2}\nabla (h-2)K_1$, where $h=|N(u^*)|>5$.
\end{lemma}
\begin{proof}
We firstly claim that $|N(u^*)|>5$. Otherwise, $\lambda x_{u^*}=\sum_{u\in N(u^*)}x_u\le 5x_{u^*}$. Therefore, $\lambda \le 5<1+\sqrt{m-2}$, a contradiction. From Lemma \ref{lem-4}, we obtain that there exists exactly one non-trivial component $H$ in $G[N(u^*)]$. It follows that
\begin{align}\label{eq5}
   \eta(H)\ge -3+\sum_{u\in N_0(u^*)}x_u/x_{u^*}\ge -3.
\end{align}
Lemma \ref{lem-1} indicated that $\eta(H)=-3$ unless $H=K_4$ or $K_5-2e$. Observe that $K_4 = F_3$, with $H_0 =\emptyset$. By a similar discussion as in Lemma \ref{lem-4}, we have $H \ne K_4$. Using the same method, we can prove $H\ne K_5-2e$. Hence the equality of (\ref{eq5}) occurs. This implies that $N_0(u^*)=\emptyset$. From the proof of Lemma \ref{lem-1}, we obtain that $H=K_2\nabla (h-2)K_1$ since $h> 5$. Therefore, we conclude that $G^*[N(u^*)]=H=K_2\nabla (h-2)K_1$.
\end{proof}

\begin{proof}[\bf Proof of Theorem \ref{thm-1}] 
According to Lemma \ref{lem-4}, we know that $G^*[N(u^*)]$ contains exactly one non-trivial component $H$.  Lemma \ref{lem-8} implies $\eta(H)=-3$ and $N_0(u^*)=\emptyset$ and $G^*[N(u^*)]=K_2\nabla (h-2)K_1$. Therefore, in order to prove Theorem \ref{thm-1}, it suffices to show that $W=\emptyset$. Suppose to the contrary that $W\ne \emptyset$. 
Note that Lemma \ref{lem-3} gives $e(W)=0$.   
In addition, we observe that \eqref{eq4} turns to be $e(W)\le \eta(H)+3=0$, and thus the equality of \eqref{eq4} holds. 
For any $w\in W$, we have $d_{N(u^*)} (w)\ge 1$ 
since $G^*$ is connected. 
The equality case of (\ref{eq4}) implies $x_w=x_{u^*}$ for $w\in W$. Let $u_1,u_2$ be the vertices with degree greater than $2$ in $G^*[N(u^*)]$, then $x_{u_i}=x_{u^*}$ since the equality of \eqref{eq'} holds. Note that \[\lambda x_{u^*}=x_{u_1}+x_{u_2}+\sum_{v\in N(u^*)\setminus \{u_1,u_2\}}x_v=2x_{u^*}+\sum_{v\in N(u^*)\setminus \{u_1,u_2\}}x_v=\lambda x_{u_i}-\sum_{w\in N_W(u_i)}x_w.\] We have $N(w)\cap \{u_1,u_2\}=\emptyset$, i.e., $N(w)\subsetneq N(u^*)$. Thus
\[\lambda x_w=\sum_{v\in N(w)}x_v<\sum_{v\in N(u^*)}x_v=\lambda x_{u^*}, \]
a contradiction. This completes the proof of Theorem \ref{thm-1}.
\end{proof}

\section*{Declaration of competing interest}
There is no competing interest.

\section*{Acknowledgments}
Lu Lu was supported by NSFC (No. 12001544) and  Natural Science Foundation of Hunan Province (No. 2021JJ40707). 
Yongtao Li is a student under the supervision of Prof. Yuejian Peng, and his work was supported by  NSFC (No. 11931002).  The authors are so grateful to the referees for their valuable comments and corrections which improve the presentation of the paper.

\end{document}